\documentclass[12pt]{amsart}
\usepackage{amsthm}
\usepackage{amssymb}\usepackage{listings}
\usepackage{longtable}
\usepackage{mathtools}
\usepackage{cite}
\usepackage{geometry}
\usepackage{enumerate}
\usepackage{setspace}
\usepackage{tikz-cd}
\usetikzlibrary{matrix, arrows}
\usepackage[unicode=true]
 {hyperref}
\hypersetup{
colorlinks=true,
urlcolor=black,
citecolor=blue,
linkcolor=blue,
}

\def\multiset#1#2{\ensuremath{\left(\kern-.3em\left(\genfrac{}{}{0pt}{}{#1}{#2}\right)\kern-.3em\right)}}

\allowdisplaybreaks
\newtheorem{thm}{Theorem}[section]
\newtheorem{prop}[thm]{Proposition}
\newtheorem{lem}[thm]{Lemma}

\theoremstyle{definition}
\newtheorem{definition}[thm]{Definition}
\newtheorem*{assumption}{Assumption}
\newtheorem{example}[thm]{Example}

\theoremstyle{remark}
\newtheorem{remark}[thm]{Remark}

\numberwithin{equation}{section}
\newcommand{\Aut}{\mathrm{Aut}}
\newcommand{\End}{\mathrm{End}}
\newcommand{\Hom}{\mathrm{Hom}}
\newcommand{\Hol}{\mathrm{Hol}}
\newcommand{\Inn}{\mathrm{Inn}}

\newcommand{\bN}{\mathbb{N}}

\begin{document}

\large 

\title[Finite skew braces with isomorphic additive and circle groups]{Finite skew braces with isomorphic non-abelian characteristically simple additive and circle groups}
\author{Cindy (Sin Yi) Tsang}
\address{Department of Mathematics\\
Ochanomizu University\\
Tokyo\\
Japan}
\email{tsang.sin.yi@ocha.ac.jp}\urladdr{http://sites.google.com/site/cindysinyitsang/} 

\date{\today}

\maketitle

\begin{abstract}A skew brace is a triplet $(A,\cdot,\circ)$, where $(A,\cdot)$ and $(A,\circ)$ are groups such that the brace relation $x\circ (y\cdot z) = (x\circ y)\cdot x^{-1}\cdot (x\circ z)$ holds for all $x,y,z\in A$. In this paper, we study the number of finite skew braces $(A,\cdot,\circ)$, up to isomorphism, such that $(A,\cdot)$ and $(A,\circ)$ are both isomorphic to $T^n$ with $T$ non-abelian simple and $n\in\bN$. We prove that it is equal to the number of unlabeled directed graphs on $n+1$ vertices, with one distinguished vertex, and whose underlying undirected graph is a tree. In particular, it depends only on $n$ and is independent of $T$.\end{abstract}

\tableofcontents

\section{Introduction}

A \emph{skew brace} is a triplet $(A,\cdot,\circ)$, where $(A,\cdot)$ and $(A,\circ)$ are groups such that the so-called brace relation
\[ x\circ(y\cdot z) = (x\circ y)\cdot x^{-1}\cdot (x\circ z)\]
holds for all $x,y,z\in A$, and here $x^{-1}$ denotes the inverse of $x$ in $(A,\cdot)$. We note that $(A,\cdot)$ and $(A,\circ)$ must have the same identity element. Often $(A,\cdot)$ is referred to as the \emph{additive} group, and $(A,\circ)$ as the \emph{circle} or \emph{multiplicative} group, of the skew brace. Despite the terminology, the group $(A,\cdot)$ need not be abelian. A skew brace $(A,\cdot,\circ)$ with abelian additive group $(A,\cdot)$ is called a \emph{brace} or sometimes a \emph{classical brace}.

\vspace{1mm}

Braces were introduced by W. Rump \cite{Rump} in order to study non-degenerate set-theoretic solutions of the Yang--Baxter equation which are involutive. L. Guarnieri and L. Vendramin \cite{Skew braces} later extended this notion to skew braces to study all such solutions which are not necessarily involutive. Enumerating all skew braces (up to isomorphism) has thus become a problem of interest. As the reader would guess, an isomorphism between two skew braces $(A_1,\cdot_1,\circ_1)$ and $(A_2,\cdot_2,\circ_2)$ is a bijective map $\phi:A_1\longrightarrow A_2$ such that
\begin{align*}\phi(x\cdot _1y) & = \phi(x)\cdot_2 \phi(y)\\\phi(x\circ_1y) & = \phi(x)\circ_2\phi(y)\end{align*}
for all $x,y,\in A_1$. In \cite{Skew braces} an algorithm to compute finite skew braces of small order is also given, which was later improved by \cite{Skew braces 2}. In \cite{Skew braces}, they also showed that there is a connection between skew braces and regular subgroups in the holomorph, as we shall explain.

\vspace{1mm}

We shall only consider skew braces with a finite underlying set. Given any finite groups $G$ and $N$ of the same order, define
\[ b(G,N) = \#\left\{\begin{array}{c}
\mbox{isomorphism classes of skew braces $(A,\cdot,\circ)$}\\\mbox{such that $(A,\cdot)\simeq N$ and $(A,\circ)\simeq G$}\end{array}\right\}.\]
Recall that a subgroup of $\mathrm{Perm}(N)$, the group of permutations of $N$, is said to be \emph{regular} if its action on $N$ is both transitive and free. For example, both $\lambda(N)$ and $\rho(N)$ are regular, where
\[\begin{cases}
\lambda: N\longrightarrow\mathrm{Perm}(N);\hspace{1em}\lambda(\eta) = (x\mapsto \eta x),\\
\rho:N\longrightarrow\mathrm{Perm}(N);\hspace{1em}\rho(\eta) = (x\mapsto x\eta^{-1}),
\end{cases}\]
are the left and right regular representations of $N$. The \emph{holomorph} of $N$ is
\[ \Hol(N) = \rho(N)\rtimes\Aut(N) = \lambda(N) \rtimes \Aut(N).\]
Let us further put
\[ \mathcal{R}(G,N) = \{\mbox{regular subgroups of $\Hol(N)$ which are isomorphic to $G$}\}.\]
Note that $\Aut(N)$ acts on this set via conjugation in $\Hol(N)$.

\begin{prop}\label{b prop} For any finite groups $G$ and $N$ of the same order, we have
\[ b(G,N) = \#\left(\mathcal{R}(G,N)/\Aut(N)\right),\]
the number of orbits of $\mathcal{R}(G,N)$ under conjugation by $\Aut(N)$.
\end{prop}
\begin{proof}See \cite[Theorem 4.2 and Proposition 4.3]{Skew braces} or \cite[Appendix]{Skew braces'}.
\end{proof}

Now, regular subgroups in the holomorph are also known to be related to Hopf--Galois structures. More specifically, define
\[ e(G,N) = \#\left\{\begin{array}{c}\mbox{Hopf--Galois structures of type $N$ on a}\\\mbox{Galois extension $L/K$ with $\mathrm{Gal}(L/K)\simeq G$}\end{array}\right\}.\]
We refer the reader to \cite[Chapter 2]{Childs book} for the definitions and background.

\begin{prop}\label{e prop} For any finite groups $G$ and $N$ of the same order, we have
\[ e(G,N) = \frac{|\Aut(G)|}{|\Aut(N)|}\cdot \#\mathcal{R}(G,N).\]
\end{prop}
\begin{proof}This follows from work of \cite{GP} and \cite{By96}, or see \cite[Chapter 2]{Childs book}.
\end{proof}

From Propositions \ref{b prop} and \ref{e prop}, we see that 
\[ b(G,N) =0\mbox{ if and only if }e(G,N) = 0.\]
However, the numbers $b(G,N)$ and $e(G,N)$ need not be equal, and the calculation of the former is more difficult in general, because one also needs to take the conjugation action of $\Aut(N)$ into account. 


\begin{assumption}In the rest of this paper, we shall take $G$ to be a finite non-abelian characteristically simple group. Equivalently, this means $G = T^n$ for some finite non-abelian simple group $T$ and $n\in\bN$.
\end{assumption}

By drawing tools from graph theory and some consequences of the classification of finite simple groups (CFSG), in \cite{Tsang PAMS} the author obtained a closed formula for the number $e(G,G)$.

\begin{thm}\label{e thm}We have the formula
\[ e(G,G) = 2^n(n|\Aut(T)+1)^{n-1}.\]
\end{thm}

For $n=1$, by Proposition \ref{e prop} and Theorem \ref{e thm}, we have
\[ e(G,G) = \#\mathcal{R}(G,G) = 2\mbox{ and so }\mathcal{R}(G,G) = \{\lambda(G),\rho(G)\}.\]
This fact was first shown in  \cite{Childs non-abelian}; also see \cite[Theorem 1.4]{Tsang HG} for generalization to quasisimple groups. From this and Proposition \ref{b prop}, we get $b(G,G)=2$.

\vspace{1mm}

For $n\geq 2$, however, the size of $\mathcal{R}(G,G)$ becomes much larger, and it is no longer obvious what the value of $b(G,G)$ ought to be simply from Theorem \ref{e thm}. Nonetheless, we are still able to determine $b(G,G)$. To state our result, we first need to introduce some definitions.

\vspace{1mm}

Recall that a \emph{tree} is a connected graph which has no cycle, or equivalently, a graph in which any two vertices can be connected by a unique simple path. It is known that a graph on $n+1$ vertices is a tree if and only if it has exactly $n$ edges and is connected. 

\begin{definition}\label{tree def}
In this paper, we shall use \emph{directed tree} to mean a directed graph whose underlying undirected graph is a tree. For brevity, put
\[ \bN_n = \{1,\dots,n\}\mbox{ and } \bN_{0,n} = \bN_n\sqcup\{0\}.\]
Let $\mathcal{T}(n)$ be the set of all labeled directed trees on $n+1$ vertices, labeled by $\bN_{0,n}$. Two directed graphs $\Gamma_1,\Gamma_2\in\mathcal{T}(n)$ are \emph{equivalent} if there is a bijection 
\[ \xi\in\mathrm{Map}(\bN_{0,n},\bN_{0,n})\mbox{ with }\xi(0)=0\]
such that for any $v,v'\in\bN_{0,n}$, there is an arrow $v\longrightarrow v'$ in $\Gamma_1$ precisely when there is an arrow $\xi(v)\longrightarrow \xi(v')$ in $\Gamma_2$. We shall write $\Gamma_1\sim \Gamma_2$ in this case. 
\end{definition}

While $e(G,G)$ depends on both $n$ and $T$, the number $b(G,G)$ turns out to depend only on $n$. Our main theorem is the following.

\begin{thm}\label{b thm}We have the equality
\[ b(G,G) = \#\left(\mathcal{T}(n)/\sim\right),\]
the number of equivalence classes of $\sim$ on the set $\mathcal{T}(n)$. \end{thm}

Observe that unlike Theorem \ref{e thm}, here we do not have a closed formula for $b(G,G)$. Let us briefly explain why this is expected to be difficult. 

\begin{remark}\label{rem}The equivalence relation $\sim$ in Definition \ref{tree def} simply means that we can forget about the labeling of the vertices except the vertex $0$. We then see that $\#(\mathcal{T}(n)/\sim)$ is equal to the size of the set
\[ \mathcal{T}^\star(n) = \left\{\begin{array}{c}
\mbox{unlabeled directed trees on $n+1$ vertices}\\\mbox{with a distinguished vertex (root)}\end{array}\right\}.\]
The underlying undirected graph of an element $\Gamma\in\mathcal{T}^\star(n)$ may be regarded as an (unordered) rooted tree.\footnote{We did not refer to elements of $\mathcal{T}^\star(n)$ as ``directed rooted trees" because it has a different meaning.} Now, the famous Cayley's formula tells us that the number of labeled trees on $n+1$ vertices is $(n+1)^{n-1}$. It was also shown in \cite{Clarke} that for each $1\leq d\leq n$, we have
\begin{equation}\label{Clarke}
\#\left\{\begin{array}{c}\mbox{labeled trees on $n+1$ vertices in}\\\mbox{which a specified vertex has degree $d$} \end{array}\right\} = {n-1 \choose d-1}n^{n-d}.\end{equation}
The formula for $e(G,G)$ in Theorem \ref{e thm} was obtained using (\ref{Clarke}). However, to compute $b(G,G)$, we need to consider unlabeled trees, and so far there is no known closed formula for the number of unlabeled (rooted) trees; see \cite[A000055 and A000081]{OEIS}. We also need to take the orientation of the directed trees into account, which adds even more complexity into the problem. It is possible that there does not even exist a closed formula for $b(G,G)$.\end{remark}

\section{Description of endomorphisms}

The endomorphisms of $G$ play an important role in the proof of both Theorems \ref{e thm} and \ref{b thm}. Following the notation in \cite{Tsang PAMS}, let us give a description of these endomorphisms.

\vspace{1mm}

For each $i\in\bN_n$, we shall write
\[ T^{(i)} = 1\times \cdots \times 1\times T\times 1\times\cdots\times 1\mbox{ ($T$ is in the $i$th position)},\]
and use  $x^{(i)}$ to denote an arbitrary element of $T^{(i)}$. For convenience, let $T^{(0)}$ denote the trivial subgroup and $x^{(0)}$ the identity element.

\begin{prop}\label{End prop}Elements of $\End(G)$ are precisely the maps
\[(x^{(1)},\dots, x^{(n)})\mapsto (\varphi_1(x^{(\theta(1))}),\dots, \varphi_n(x^{(\theta(n))})),\]
where $\theta\in\mathrm{Map}(\bN_n,\bN_{0,n})$, and $\varphi_i\in\Hom(T^{(\theta(i))},T^{(i)})$ for each $i\in\bN_n$. Moreover, the above map lies in $\Aut(G)$ precisely when $\theta$ is injective with image equal to $\bN_n$ and $\varphi_i$ is bijective for each $i\in\bN_n$.
\end{prop}
\begin{proof}This basically follows from the proof of \cite[Lemma 3.2]{Byott simple}, which gives a description of elements of $\Aut(G)$.
\end{proof}

In view of Proposition \ref{End prop}, we introduce the following notation.

\begin{definition}\label{End def}Given any $f\in\End(G)$, we shall write
\[ f(x^{(1)},\dots,x^{(n)}) = (\varphi_{f,1}(x^{(\theta_f(1))}),\dots,\varphi_{f,n}(x^{(\theta_f(n))})),\]
where $\theta_f\in\mathrm{Map}(\bN_n,\bN_{0,n})$, and $\varphi_{f,i}\in\Hom(T^{(\theta_f(i))},T^{(i)})$ for each $i\in\bN_n$. To ensure that $\theta_f$ is uniquely determined by $f$, we shall assume that
\begin{equation}\label{phi cond} \varphi_{f,i}\mbox{ is non-trivial when }\theta_f(i)\neq0.\end{equation}
Note that in this case $\varphi_{f,i}$ is necessarily bijective because $T$ is simple. Let us define $\theta_f(0) = 0$, so that we may regard $\theta_f\in\mathrm{Map}(\bN_{0,n},\bN_{0,n})$. Also, we shall write $\varphi_{f,0}\in \Hom(T^{(0)},T^{(0)})$ for the trivial map for convenience.
\end{definition}

\section{Characterization of regular subgroups}

In \cite{Tsang PAMS}, the present author computed $e(G,G)$ by first giving a characterization of $\mathcal{R}(G,G)$ in terms of trees. Let us recall this characterization.

\vspace{1mm}

Using a combinatorial argument, it was shown in \cite[Section 3]{Tsang PAMS} that
\begin{equation}\label{Out fact}\mbox{every element of $\mathcal{R}(G,G)$ lies in $\mathrm{InHol}(G) = \rho(G)\rtimes\mathrm{Inn}(G)$},\end{equation}
where $\Inn(G)$ denotes the inner automorphism group of $G$. We remark that the proof of (\ref{Out fact}) requires the fact that $T$ has solvable outer automorphism group, which is known as Schreier conjecture and is a consequence of CFSG. 

\vspace{1mm}

Recall that a pair $(f,g)$, with $f,g\in\End(G)$, is \emph{fixed point free} (fpf) if 
\[f(x) = g(x)\mbox{ holds exactly when }x = 1.\]
Given any $f,g\in\End(G)$, we can define a subgroup of $\mathrm{InHol}(G)$ by setting
\begin{equation}\label{G fg}\mathcal{G}_{(f,g)} = \{\rho(g(x))\lambda(f(x)):x\in G\}.\end{equation}
It is easy to check that $\mathcal{G}_{(f,g)}$ is regular, in which case $\mathcal{G}_{(f,g)}\simeq G$, if and only if $(f,g)$ is fpf; see \cite[Proposition 1]{Byott Childs}. Since $G$ has trivial center, these are all the regular subgroups isomorphic to $G$ in $\mathrm{InHol}(G)$ by\cite[Proposition 6]{Byott Childs}, so
\begin{equation}\label{R fpf}\mathcal{R}(G,G) = \{\mathcal{G}_{(f,g)}: f,g\in\End(G)\mbox{ such that $(f,g)$ is fpf}\}.\end{equation}
We are thus reduced to the problem of determining when a pair $(f,g)$, with $f,g\in \End(G)$, is fpf.  Plainly $(f,g)$ is fpf if and only if the equations
\begin{equation}\label{ei} \varphi_{f,i}(x^{(\theta_f(i))}) = \varphi_{g,i}(x^{(\theta_g(i))}),\end{equation}
for $i$ ranging over $\bN_n$, have no common solution except 
\[x^{(1)} = \cdots = x^{(n)} = 1.\]
Here the notation is as in Definition \ref{End def}. To study this system of equations, in \cite{Tsang PAMS} the author developed a graph-theoretic method, as follows. 
\begin{definition}\label{graph fg def}Given any $f,g\in\End(G)$, define $\Gamma_{\{f,g\}}$ to be the undirected multigraph with vertex set $\bN_{0,n}$, and we draw one edge $\mathfrak{e}_i$ between $\theta_f(i)$ and $\theta_g(i)$ for each $i\in\bN_{n}$. Note that $\Gamma_{\{f,g\}}$ has $n+1$ vertices and $n$ edges, so it is a tree precisely when it is connected.
 \end{definition}

The vertices of $\Gamma_{\{f,g\}}$ should be viewed as the identity element $x^{(0)}$ and the variables $x^{(1)},\dots,x^{(n)}$. For each $i\in\bN_n$, the edge $\mathfrak{e}_i$ represents (\ref{ei}), and
\begin{equation}\label{observation1}
\begin{cases}
\mbox{(\ref{ei}) is equivalent to $x^{(\theta_f(i))} = (\varphi_{f,i}^{-1}\circ\varphi_{g,i})(x^{(\theta_g(i))})$}&\mbox{if }\theta_f(i)\neq0,\\
\mbox{(\ref{ei}) is equivalent to $x^{(\theta_g(i))} = (\varphi_{g,i}^{-1}\circ\varphi_{f,i})(x^{(\theta_f(i))})$}&\mbox{if }\theta_g(i)\neq0,
\end{cases}
\end{equation}
by condition (\ref{phi cond}). Thus, we can solve the equations (\ref{ei}) by ``following the edges'' in the graph $\Gamma_{\{f,g\}}$. Let us illustrate the idea via two examples.

\begin{example}\label{ex1}Take $n=3$ and let $f,g\in\End(G)$ be such that
\begin{align*}
f(x^{(1)},x^{(2)},x^{(3)}) & = (\varphi_{f,1}(x^{(2)}),\varphi_{f,2}(x^{(0)}),\varphi_{f,3}(x^{(1)})),\\
g(x^{(1)},x^{(2)},x^{(3)}) & = (\varphi_{g,1}(x^{(1)}),\varphi_{g,2}(x^{(3)}),\varphi_{g,3}(x^{(0)})).
\end{align*}
The associated graph $\Gamma_{\{f,g\}}$ is given by
\[\begin{tikzpicture}
\node at (0,0) [name = 0] {$0$};
\node at (-1.5,0) [name = 1] {$1$};
\node at (-3,0) [name = 2]{$2$};
\node at (1.5,0) [name = 3]{$3$};
\draw[-] (0) -- (1) -- (2) (0) -- (3);
\end{tikzpicture}\]
and is a tree. If $f(x^{(1)},x^{(2)},x^{(3)}) = g(x^{(1)},x^{(2)},x^{(3)})$ holds, then based on the observations in (\ref{observation1}), we see that: 
\begin{itemize}
\item The edge $0-1$ tells us that $x^{(1)} = (\varphi_{f,3}^{-1}\circ\varphi_{g,3})(x^{(0)})=1$. 
\item The edge $0-3$ tells us that $x^{(3)} = (\varphi_{g,2}^{-1}\circ\varphi_{f,2})(x^{(0)}) = 1$.
\item The edge $1-2$ tells us that $x^{(2)}$ is determined by $x^{(1)}$ with
\[x^{(2)} = (\varphi_{f,1}^{-1}\circ\varphi_{g,1})(x^{(1)}).\]
Since we already know that $x^{(1)}=1$, this implies that $x^{(2)}=1$ also. 
\end{itemize}
We conclude that the pair $(f,g)$ is fpf.
\end{example}

\begin{example}\label{ex2}Take $n=3$ and let $f,g\in\End(G)$ be such that
\begin{align*}
f(x^{(1)},x^{(2)},x^{(3)}) & = (\varphi_{f,1}(x^{(0)}),\varphi_{f,2}(x^{(2)}),\varphi_{f,3}(x^{(2)})),\\
g(x^{(1)},x^{(2)},x^{(3)}) & = (\varphi_{g,1}(x^{(1)}),\varphi_{g,2}(x^{(3)}),\varphi_{g,3}(x^{(3)})).
\end{align*}
The associated graph $\Gamma_{\{f,g\}}$ is given by
\[\begin{tikzpicture}
\node at (0,0) [name = 0] {$0$};
\node at (1.5,0) [name = 1] {$1$};
\node at (3,0) [name = 2]{$2$};
\node at (4.5,0) [name = 3]{$3$};
\draw[-] (0) -- (1);
\draw[-] (2) edge [bend right = 30] (3);
\draw[-] (2) edge [bend left = 30] (3);
\end{tikzpicture}\]
and is not a tree. The cycle $2-3-2$ may be thought of as representing
\[ (\varphi_{f,3}^{-1}\circ\varphi_{g,3}\circ\varphi_{g,2}^{-1}\circ\varphi_{f,2})(x^{(2)}) = x^{(2)}.\]
It is a consequence of CFSG that $T$ has no fpf automorphism, whence there exists $\sigma^{(2)}\in T^{(2)}$ with $\sigma^{(2)}\neq1$ such that
\[ (\varphi_{f,3}^{-1}\circ\varphi_{g,3}\circ\varphi_{g,2}^{-1}\circ\varphi_{f,2})(\sigma^{(2)}) = \sigma^{(2)}.\]
We then see that $(f,g)$ has a fixed point other than the identity, namely
\[ (x^{(1)},x^{(2)},x^{(3)}) = (1,\sigma^{(2)},(\varphi_{g,2}^{-1}\circ\varphi_{f,2})(\sigma^{(2)})).\]
We conclude that the pair $(f,g)$ is not fpf.
\end{example}

The above ideas may be formalized, and we have the following criterion.

\begin{prop}\label{criterion prop}A pair $(f,g)$, with $f,g\in\End(G)$, is fpf if and only if the graph $\Gamma_{\{f,g\}}$ is a tree.
\end{prop}
\begin{proof}See \cite[Propositions 2.5 and 2.9]{Tsang PAMS}. \end{proof}

Combining (\ref{R fpf}) and Proposition \ref{criterion prop}, we obtain
\begin{equation}\label{R graph}
\mathcal{R}(G,G) = \{\mathcal{G}_{(f,g)}: f,g\in\End(G)\mbox{ such that $\Gamma_{\{f,g\}}$ is tree}\}.\end{equation}
We remark that the proof of Theorem \ref{e thm} given in \cite{Tsang PAMS} uses this characterization of $\mathcal{R}(G,G)$ and the formula (\ref{Clarke}) with $0$ as the specified vertex. 

\section{Counting orbits of regular subgroups}

In this section, we shall prove Theorem \ref{b thm}. To do so, we need to consider the conjugation action of $\Aut(G)$ on elements of $\mathcal{R}(G,G)$, and we shall again use graph theory. Let us first refine Definition \ref{graph fg def} as follows.

\begin{definition}\label{Gamma def}Given any $f,g\in\End(G)$, define $\Gamma_{(f,g),\rightarrow}$ to be the directed multigraph obtained from $\Gamma_{\{f,g\}}$ by replacing the edge $\mathfrak{e}_i$ with an arrow from $\theta_f(i)$ to $\theta_g(i)$ for each $i\in\bN_{n}$. The arrow $\rightarrow$ in the subscript indicates that the orientation is chosen to go from $\theta_f(\cdot)$ to $\theta_g(\cdot)$. 
\end{definition}


Recall the definition in (\ref{G fg}) and also the equivalence relation $\sim$ on the set $\mathcal{T}(n)$ in Definition \ref{tree def}. The key is the next theorem.

\begin{thm}\label{key thm}Let $(f,g)$ and $(f',g')$, with $f,g,f',g'\in\End(G)$, be such that the graphs $\Gamma_{\{f,g\}}$ and $\Gamma_{\{f',g'\}}$ are trees. Then, the subgroups $\mathcal{G}_{(f,g)}$ and $\mathcal{G}_{(f',g')}$ are $\Aut(G)$-conjugates exactly when $\Gamma_{(f,g),\rightarrow}\sim\Gamma_{(f',g'),\rightarrow}$.
\end{thm}

From (\ref{R graph}) and Theorem \ref{key thm}, we see that $\mathcal{G}_{(f,g)}\mapsto\Gamma_{(f,g),\rightarrow}$, induces a well-defined injection
\[ \{\Aut(G)\mbox{-orbits of }\mathcal{R}(G,G)\}\longrightarrow \{\mbox{equivalence classes of $\sim$ on $\mathcal{T}(n)$}\}.\]
This map is also surjective. Indeed, given any $\Gamma$ in $\mathcal{T}(n)$, label the $n$ arrows in $\Gamma$ as $\mathfrak{a}_1,\dots,\mathfrak{a}_n$, where the numbering may be randomly chosen. For each $i\in\bN_n$, simply define $\theta_f(i)$ to be the tail and $\theta_g(i)$ to be the head of $\mathfrak{a}_i$. Then, clearly $\Gamma = \Gamma_{(f,g),\rightarrow}$ for $f,g\in\End(G)$ given by
\begin{align*}
f(x^{(1)},\dots,x^{(n)}) & = (x^{(\theta_f(1))},\dots,x^{(\theta_f(n))}),\\
g(x^{(1)},\dots,x^{(n)}) & = (x^{(\theta_g(1))},\dots,x^{(\theta_g(n))}).
\end{align*}
Theorem \ref{b thm} now follows from Proposition \ref{b prop}.

\vspace{1mm}

In the next two subsections, we shall prove Theorem \ref{key thm}, from which Theorem \ref{b thm} would follow. Before we proceed, let us make two observations.

\begin{lem}\label{observation lem}Let $(f,g)$ and $(f',g')$, with $f,g\in\End(G)$, be such that $\Gamma_{\{f,g\}}$ and $\Gamma_{\{f',g'\}}$ are trees. Then $\mathcal{G}_{(f,g)}$ and $\mathcal{G}_{(f',g')}$ are $\Aut(G)$-conjugates precisely when there exist $\psi,\pi\in\Aut(G)$ such that
\begin{equation}\label{conj eqn}(f',g') = (\psi\circ f\circ \pi,\psi\circ g\circ \pi).\end{equation}
\end{lem}
\begin{proof}Recall that $(f,g)$ and $(f',g')$ are fpf by Proposition \ref{criterion prop}, so then $\mathcal{G}_{(f,g)}$ and $\mathcal{G}_{(f',g')}$ are regular subgroups. They are $\Aut(G)$-conjugates if and only if there exist $\psi\in\Aut(G)$ and $\pi\in\mathrm{Perm}(G)$ such 
that for all $x\in G$, we have \begin{align}\label{fg eqn}
\rho(g'(x))\cdot\lambda(f'(x)) & = \psi\rho(g(\pi(x)))\psi^{-1}\cdot\psi\lambda(f(\pi(x)))\psi^{-1}\\\notag
& =\rho((\psi\circ g\circ \pi)(x))\cdot\lambda((\psi\circ f\circ \pi)(x)).
\end{align}
Note that $\lambda(G)$ and $\rho(G)$ intersect trivially because $G$ has trivial center. It then follows that the equation (\ref{fg eqn}) holds for all $x\in G$ exactly when (\ref{conj eqn}) is satisfied. Thus, it remains to prove that $\pi$ must be a homomorphism in this case. Since $\psi\in\Aut(G)$, from (\ref{conj eqn}) we deduce that $f\circ\pi$ and $g\circ\pi$ are both homomorphisms, so for all $x,y\in G$, we have
\[ \pi(xy) \equiv \pi(x) \pi(y)\hspace{-2mm}\pmod{\ker(f)\cap\ker(g)}.\]
But $(f,g)$ is fpf, whence $\ker(f)$ and $\ker(g)$ intersect trivially. We see that $\pi$ is indeed a homomorphism. 
\end{proof}

\begin{lem}\label{observation lem'}Let $f\in \End(G)$ and $\psi,\pi\in\Aut(G)$. Then, we have
\[ \theta_{\psi\circ f\circ\pi} = \theta_\pi\circ\theta_f\circ\theta_\psi\mbox{ in }\mathrm{Map}(\bN_{0,n},\bN_{0,n}),\]
and for each $i\in\bN_{n}$, we also have
\[ \varphi_{\psi\circ f\circ \pi,i} = \varphi_{\psi,i}\circ\varphi_{f,\theta_\psi(i)}\circ\varphi_{\pi,(\theta_f\circ\theta_\psi)(i)}.\]
Here the notation is as in Definition \ref{End def}.
\end{lem}
\begin{proof}We compute that
\begin{align*}(x^{(1)},\dots, x^{(n)}) 
& \xmapsto{\hspace{2mm}\pi\hspace{2mm}} (\dots,\varphi_{\pi,i}(x^{(\theta_\pi(i))}),\dots)\\
& \xmapsto{\hspace{2mm}f\hspace{2mm}} (\dots, (\varphi_{f,i}\circ\varphi_{\pi,\theta_f(i)})(x^{((\theta_\pi\circ\theta_f)(i))}),\dots)\\
& \xmapsto{\hspace{2mm}\psi\hspace{2mm}} (\dots, (\varphi_{\psi,i}\circ\varphi_{f,\theta_\psi(i)}\circ\varphi_{\pi,(\theta_f\circ\theta_\psi)(i)})(x^{((\theta_\pi\circ\theta_f\circ\theta_\psi)(i))}),\dots),
\end{align*}
and from here, the claims are clear once we verify the condition (\ref{phi cond}). More precisely, we need to check that for each $i\in\bN_{n}$, we have
\begin{equation}\label{phi cond'}  \varphi_{\psi,i}\circ\varphi_{f,\theta_\psi(i)}\circ\varphi_{\pi,(\theta_f\circ\theta_\psi)(i)}\mbox{ is non-trivial when }(\theta_\pi\circ\theta_f\circ\theta_\psi)(i)\neq0.\end{equation}
Since $\psi,\pi\in\Aut(G)$, the maps $\theta_\pi$ and  $\varphi_{\psi,i}$ are bijections, and $\varphi_{\pi,(\theta_f\circ\theta_\psi)(i)}$ is a bijection when $(\theta_f\circ\theta_\psi)(i)\neq0$. It follows that (\ref{phi cond'}) is equivalent to
\[ \varphi_{f,\theta_\psi(i)} \mbox{ is non-trivial when }\theta_f(\theta_\psi(i))\neq0,\]
and this indeed holds by (\ref{phi cond}).
\end{proof}

\begin{assumption}
In the rest of this paper, we shall fix pairs $(f,g)$ and $(f',g')$, with $f,g,f',g'\in\End(G)$, to be such that $\Gamma_{\{f,g\}}$ and $\Gamma_{\{f',g'\}}$ are trees.
\end{assumption}

\subsection{Proof of Theorem \ref{key thm}: forward implication} This direction is easy to prove. Suppose that $\mathcal{G}_{(f,g)}$ and $\mathcal{G}_{(f',g')}$ are $\Aut(G)$-conjugates. Then
\[ (\theta_{f'},\theta_{g'}) = (\theta_\pi\circ\theta_f\circ\theta_\psi,\theta_\pi\circ\theta_g\circ\theta_\psi)\mbox{ for some }\psi,\pi\in\Aut(G)\]
by Lemmas \ref{observation lem} and \ref{observation lem'}. Since $\psi,\pi\in\Aut(G)$, the maps
\[ \theta_\psi,\theta_\pi:\bN_{0,n}\longrightarrow \bN_{0,n}\mbox{ are bijections, with }\theta_\psi(0)=0 =\theta_\pi(0).\]
For any $v,v'\in\bN_{0,n}$, we then see from Definition \ref{Gamma def} that
\begin{align*}
&\hspace{1.5cm}\mbox{there is an arrow $v\longrightarrow v'$ in $\Gamma_{(f,g),\rightarrow}$}\\
&\iff (v,v') = (\theta_f(i),\theta_g(i))\mbox{ for some }i\in\bN_{n}\\
&\iff (v,v') = (\theta_f(\theta_\psi(j)),\theta_g(\theta_\psi(j)))\mbox{ for some }j\in\bN_n\\
&\iff (\theta_\pi(v),\theta_\pi(v')) = ((\theta_\pi\circ\theta_f\circ\theta_\psi)(j),(\theta_\pi\circ\theta_g\circ\theta_\psi)(j))\mbox{ for some }j\in\bN_n\\
&\iff (\theta_\pi(v),\theta_\pi(v')) = (\theta_{f'}(j),\theta_{g'}(j))\mbox{ for some }j\in\bN_n\\
&\iff\mbox{there is an arrow $\theta_\pi(v)\longrightarrow \theta_\pi(v')$ in $\Gamma_{(f',g'),\rightarrow}$}.
\end{align*}
It thus follows that $\Gamma_{(f,g),\rightarrow}$ and $\Gamma_{(f',g'),\rightarrow}$ are equivalent via $\theta_\pi$. This proves the forward implication of Theorem \ref{key thm}.

\vspace{1mm}
\subsection{Proof of Theorem \ref{key thm}: backward implication}

This direction is a lot harder to prove. Suppose that $\Gamma_{(f,g),\rightarrow}$ and $\Gamma_{(f',g'),\rightarrow}$ are equivalent, which means that there exists a bijection
\[ \xi\in\mathrm{Map}(\bN_{0,n},\bN_{0,n})\mbox{ with }\xi(0) = 0\]
such that for any $v,v'\in\bN_{0,n}$, there is an arrow $v\longrightarrow v'$ in $\Gamma_{(f,g),\rightarrow}$ precisely when there is an arrow $\xi(v)\longrightarrow \xi(v')$ in $\Gamma_{(f',g'),\rightarrow}$. By Lemmas \ref{observation lem} and \ref{observation lem'}, we need to show that there exist $\psi,\pi\in\Aut(G)$ such that
\begin{equation}\label{cond1} (\theta_{f'},\theta_{g'}) = (\theta_\pi\circ\theta_f\circ\theta_\psi,\theta_\pi\circ\theta_g\circ \theta_\psi),\end{equation}
and for each $i\in\bN_n$, we have the equalities
\begin{equation}\label{cond2}
\varphi_{f',i}  = \varphi_{\psi,i}\circ\varphi_{f,\theta_\psi(i)}\circ\varphi_{\pi,(\theta_f\circ\theta_\psi)(i)},\,\
\varphi_{g',i}  = \varphi_{\psi,i}\circ\varphi_{g,\theta_\psi(i)}\circ\varphi_{\pi,(\theta_g\circ\theta_\psi)(i)}.\end{equation}
In other words, we need to find bijections
\[ \theta_\psi,\theta_\pi\in\mathrm{Map}(\bN_{0,n},\bN_{0,n})\mbox{ with }\theta_\psi(0) = 0 = \theta_\pi(0)\]
satisfying (\ref{cond1}), as well as isomorphisms
\begin{equation}\label{solution}\varphi_{\psi,i}\in\mathrm{Iso}(T^{(\theta_\psi(i))},T^{(i)}),\, \varphi_{\pi,i}\in\mathrm{Iso}(T^{(\theta_\pi(i))},T^{(i)})\mbox{ for }i\in\bN_n\end{equation}
satisfying both equations in (\ref{cond2}).

\vspace{1mm}

First, we consider condition (\ref{cond1}). Let us take $\theta_\pi = \xi$, and we shall define $\theta_\psi$ via the next proposition. It shall be helpful to recall Definition \ref{Gamma def}.

\begin{prop}\label{theta psi def}For each $i\in\bN_{n}$, there is a unique $\theta_\psi(i)\in\bN_{n}$ such that
\[ (\theta_\pi^{-1}(\theta_{f'}(i)), \theta_\pi^{-1}(\theta_{g'}(i))) = (\theta_f(\theta_\psi(i)),\theta_g(\theta_\psi(i))). \]
Moreover, the map
\[ \theta_\psi\in\mathrm{Map}(\bN_{0,n}, \bN_{0,n}),\mbox{ where we define }\theta_\psi(0) = 0,\]
is a bijection.
\end{prop}
\begin{proof}For $i\in \bN_n$, there is an arrow $\theta_{f'}(i)\longrightarrow\theta_{g'}(i)$ in $\Gamma_{(f',g'),\rightarrow}$ by definition, so there is an arrow $\theta_\pi^{-1}(\theta_{f'}(i))\longrightarrow \theta_\pi^{-1}(\theta_{g'}(i))$ in $\Gamma_{(f,g),\rightarrow}$. This means that
\[ (\theta_\pi^{-1}(\theta_{f'}(i)), \theta_\pi^{-1}(\theta_{g'}(i))) = (\theta_f(j_i),\theta_g(j_i))\]
for some $j_i\in\bN_n$, which is unique, for otherwise there would be two distinct edges joining $\theta_\pi^{-1}(\theta_{f'}(i))$ and $\theta_\pi^{-1}(\theta_{g'}(i))$ in $\Gamma_{\{f,g\}}$, and $\Gamma_{\{f,g\}}$ would not be a tree. Similarly distinct $i$ give rise to distinct $j_i$ because $\Gamma_{\{f',g'\}}$ is a tree. The claims are then clear by setting $\theta_\psi(i) =j_i$.\end{proof}

We have thus chosen bijections $\theta_\psi$ and $\theta_\pi$ such that (\ref{cond1}) is satisfied. 
Now, we also need to show that the equations (\ref{cond2}), for $i$ ranging over $\bN_n$, have a common solution (\ref{solution}). Let us put
\begin{align*}
\bN_{n,\mathbf{f}} & = \{ i \in \bN_n : \theta_{f'}(i) = 0\} = \{i\in \bN_n:\theta_f(\theta_\psi(i)) = 0\},\\
\bN_{n,\mathbf{g}} & = \{i \in\bN_n : \theta_{g'}(i) = 0\} = \{i\in\bN_n:\theta_g(\theta_\psi(i)) = 0\},
\end{align*}
where the latter equalities hold by (\ref{cond1}) and the fact that $\theta_\pi$ is bijection with $\theta_\pi(0) = 0$. By Definition \ref{End def}, we have
\[\begin{cases}
 \varphi_{f',i} \in \Hom(T^{(\theta_{f'}(i))},T^{(i)}),\,\ \varphi_{f,\theta_\psi(i)}\in \Hom(T^{(\theta_f(\theta_\psi(i)))},T^{(\theta_\psi(i))}),\\
  \varphi_{g',i} \in \Hom(T^{(\theta_{g'}(i))},T^{(i)}),\,\ \varphi_{g,\theta_\psi(i)}\in \Hom(T^{(\theta_g(\theta_\psi(i)))},T^{(\theta_\psi(i))}).\end{cases}\]
Thus, the first and second equations in (\ref{cond2}), respectively, trivially hold for $i\in\bN_{n,\mathbf{f}}$ and $i\in\bN_{n,\mathbf{g}}$. So we only need to find (\ref{solution}) which satisfies
\begin{align}\label{cond2f}
\varphi_{f',i} & = \varphi_{\psi,i}\circ\varphi_{f,\theta_\psi(i)}\circ\varphi_{\pi,(\theta_f\circ\theta_\psi)(i)} \mbox{ for }i\in\bN_n\setminus\bN_{n,\mathbf{f}}, \\\label{cond2g}
\varphi_{g',i} & = \varphi_{\psi,i}\circ\varphi_{g,\theta_\psi(i)} \circ\varphi_{\pi,(\theta_g\circ\theta_\psi)(i)}\mbox{ for }i\in\bN_n\setminus\bN_{n,\mathbf{g}}.
\end{align}
We are in a situation similar to (\ref{ei}). Again, we shall use graph theory and represent these equations as edges on a graph, as follows.

\begin{definition}\label{Gamma bipartite}Let us put
\[ \Psi = \{1_\psi,\dots,n_\psi\}\mbox{ and }\Pi = \{1_\pi,\dots,n_\pi\}.\]
Define $\Gamma_{\Psi,\Pi}$ to be the undirected simple graph with vertex set $\Psi\sqcup\Pi$, and we draw one edge $\mathfrak{e}_{i,\mathbf{f}}$ between $i_\psi$ and $(\theta_f\circ\theta_\psi)(i)_\pi$ for each $i\in\bN_n\setminus\bN_{n,\mathbf{f}}$, as well as one edge $\mathfrak{e}_{i,\mathbf{g}}$ between $i_\psi$ and $(\theta_g\circ\theta_\psi)(i)_\pi$ for each $i\in\bN_n\setminus\bN_{n,\mathbf{g}}$. For all $i\in\bN_n$,  we have $\theta_\psi(i) \in \bN_n$, and so $(\theta_f\circ\theta_\psi)(i)\neq (\theta_g\circ\theta_\psi)(i)$ because $\Gamma_{\{f,g\}}$ has no loop. This means that $\Gamma_{\Psi,\Pi}$ is indeed a simple graph. Also $\Gamma_{\Psi,\Pi}$ is a bipartite graph, with parts $\Psi$ and $\Pi$.
\end{definition}

\begin{definition}\label{def1}For each $i\in\bN_{n}$, define 
\[ \begin{cases}
\Phi_{i,\mathbf{f}}:\mathrm{Iso}(T^{(\theta_\psi(i))},T^{(i)})\longrightarrow\mathrm{Iso}(T^{(\theta_{f'}(i))},T^{((\theta_f\circ\theta_\psi)(i))})&\mbox{ when }i\notin\bN_{n,\mathbf{f}},\\
\Phi_{i,\mathbf{g}}:\mathrm{Iso}(T^{(\theta_\psi(i))},T^{(i)})\longrightarrow\mathrm{Iso}(T^{(\theta_{g'}(i))},T^{((\theta_g\circ\theta_\psi)(i))})&\mbox{ when }i\notin\bN_{n,\mathbf{g}},
\end{cases}\]
respectively, by setting
\[\begin{cases}\Phi_{i,\mathbf{f}}(\varphi) = \varphi_{f,\theta_\psi(i)}^{-1}\circ\varphi^{-1}\circ\varphi_{f',i} &\mbox{ when }i\notin\bN_{n,\mathbf{f}},\\
\Phi_{i,\mathbf{g}}(\varphi) = \varphi_{g,\theta_\psi(i)}^{-1}\circ\varphi^{-1}\circ\varphi_{g',i}&\mbox{ when }i\notin\bN_{n,\mathbf{g}}.
\end{cases}\]
Note that these $\Phi_{i,\mathbf{f}}$ and $\Phi_{i,\mathbf{g}}$ are well-defined bijections by condition (\ref{phi cond}).
\end{definition}

The vertices in $\Psi$ and $\Pi$, respectively, should be viewed as the variables
\[ \varphi_{\psi,1},\dots,\varphi_{\psi,n}\mbox{ and }\varphi_{\pi,1},\dots,\varphi_{\pi,n}\]
which we wish to solve. For each $i\in\bN_n$, the edges $\mathfrak{e}_{i,\mathbf{f}}$ and $\mathfrak{e}_{i,\mathbf{g}}$, respectively, represent (\ref{cond2f}) and (\ref{cond2g}) when $i\not\in\bN_{n,\mathbf{f}}$ and $i\not\in\bN_{n,\mathbf{g}}$. Observe that:
\begin{equation}\label{observation2}
\begin{cases}
\mbox{(\ref{cond2f}) is equivalent to $\varphi_{\pi,(\theta_f\circ\theta_\psi)(i)} = \Phi_{i,\mathbf{f}}(\varphi_{\psi,i})$}&\mbox{ for }i\not\in\bN_{n,\mathbf{f}},\\
\mbox{(\ref{cond2g}) is equivalent to $\varphi_{\pi,(\theta_g\circ\theta_\psi)(i)} = \Phi_{i,\mathbf{g}}(\varphi_{\psi,i})$}&\mbox{ for }i\not\in\bN_{n,\mathbf{g}}.
\end{cases}\end{equation}
Therefore, analogous to the case of (\ref{ei}), we may solve the equations (\ref{cond2f}) and (\ref{cond2g}) by ``following the edges" in the graph $\Gamma_{\Psi,\Pi}$. Let us first consider an explicit example to illustrate the idea.

\begin{example}\label{ex}Take $n=3$ and suppose that
\begin{align*}
f(x^{(1)},x^{(2)},x^{(3)}) & =  (\varphi_{f,1}(x^{(1)}),\varphi_{f,2}(x^{(3)}),\varphi_{f,3}(x^{(2)})),\\
g(x^{(1)},x^{(2)},x^{(3)}) & = (\varphi_{g,1}(x^{(0)}),\varphi_{g,2}(x^{(1)}),\varphi_{g,3}(x^{(1)})),\\
f'(x^{(1)},x^{(2)},x^{(3)}) & =  (\varphi_{f',1}(x^{(1)}),\varphi_{f',2}(x^{(2)}),\varphi_{f',3}(x^{(3)})),\\
g'(x^{(1)},x^{(2)},x^{(3)}) & = (\varphi_{g',1}(x^{(2)}),\varphi_{g',2}(x^{(0)}),\varphi_{g',3}(x^{(2)})).
\end{align*}
The graphs $\Gamma_{(f,g),\rightarrow}$ and $\Gamma_{(f',g'),\rightarrow}$, respectively, are equal to
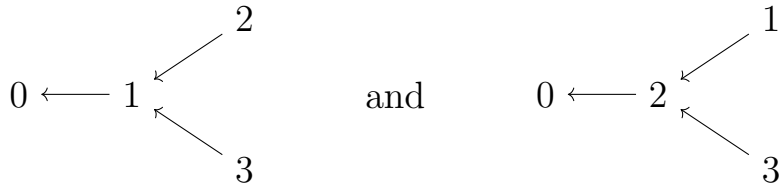
\[\begin{tikzpicture}
\node at (0,0) [name = 0] {$0$};
\node at (1.5,0) [name = 1] {$1$};
\node at (3,1) [name = 2] {$2$};
\node at (3,-1) [name = 3] {$3$};
\draw[->] (1) edge (0) (2) edge (1) (3) edge (1);
\node at (5,0) {and};
\node at (7,0) [name = 0'] {$0$};
\node at (8.5,0) [name = 2'] {$2$};
\node at (10,1) [name = 1'] {$1$};
\node at (10,-1) [name = 3'] {$3$};
\draw[->] (2') edge (0') (1') edge (2') (3') edge (2');
\end{tikzpicture}\]
according to Definition \ref{Gamma def}. We may take $\xi = \theta_\pi$ to be the bijection
\[ \theta_\pi(0) = 0,\,\ \theta_\pi(1) = 2,\,\ \theta_\pi(2) = 1,\,\ \theta_\pi(3) = 3,\]
and following Proposition \ref{theta psi def}, we see that $\theta_\psi$ in turn is given by
\[ \theta_\psi(0) = 0,\,\ \theta_\psi(1) = 3,\,\ \theta_\psi(2) = 1,\,\ \theta_\psi(3) = 2. \]
Then, the graph $\Gamma_{\Psi,\Pi}$ is equal to
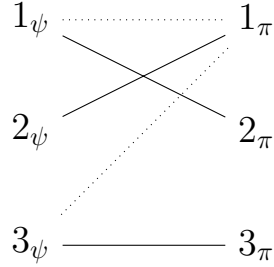
\[\begin{tikzpicture}
\node at (0,3) [name = 1] {$1_\psi$};
\node at (0,1.5) [name = 2] {$2_\psi$};
\node at (0,0) [name = 3] {$3_\psi$};
\node at (3,3) [name = 1'] {$1_\pi$};
\node at (3,1.5) [name = 2'] {$2_\pi$};
\node at (3,0) [name = 3'] {$3_\pi$}; 
\draw[-] (1) -- (2') (2) -- (1') (3) -- (3');
\draw[dotted,-] (1) -- (1') (3) -- (1');
\end{tikzpicture}\]
where we have used solid and dotted lines, respectively, to denote the edges $\mathfrak{e}_{i,\mathbf{f}}$ and $\mathfrak{e}_{i,\mathbf{g}}$. To get a solution to (\ref{cond2f}) and (\ref{cond2g}), we can first pick 
\[ \varphi_{\psi,1}\in \mathrm{Iso}(T^{(\theta_\psi(1))},T^{(1)})\]
to be any isomorphism. Then, we ``follow the edges" in $\Gamma_{\Psi,\Pi}$ and define
\[\begin{cases}
 \varphi_{\pi,1} = \Phi_{1,\mathbf{g}}(\varphi_{\psi,1}),\,\ \varphi_{\pi,2}  = \Phi_{1,\mathbf{f}}(\varphi_{\psi,1}).\\
 \varphi_{\psi,2}  = \Phi_{2,\mathbf{f}}^{-1}(\varphi_{\pi,1}),\,\ \varphi_{\psi,3}  = \Phi_{3,\mathbf{g}}^{-1}(\varphi_{\pi,1}),\\
 \varphi_{\pi,3} = \Phi_{3,\mathbf{f}}(\varphi_{\psi,3}),\\
\end{cases}\]
recursively. The equations (\ref{cond2f}) and (\ref{cond2g}) are then all satisfied by (\ref{observation2}). We have thus obtained a common solution (\ref{solution}) to the equations (\ref{cond2}).\end{example}

The idea of Example \ref{ex} is somewhat similar to that of Example \ref{ex1}. But unlike Proposition \ref{criterion prop}, in this case we do not need $\Gamma_{\Psi,\Pi}$ to be a tree for the idea to work in general. Indeed, in the special case that
\[ f' = f = \mathrm{Id}_G,\,\ g' = g = \mbox{the trivial map},\, \theta_\pi = \mbox{Id}_{\bN_{0,n}} = \theta_{\psi},\]
the graph $\Gamma_{\Psi,\Pi}$ has exactly $n$ edges, namely
\[ \begin{tikzpicture}
\node at (0,0) [name =1] {$i_\psi$};
\node at (3,0) [name =1'] {$i_\pi$};
\path[-] (1) edge node [above] {$\mathfrak{e}_{i,\mathbf{f}}$} (1');
\node at (5,0) {for $i\in\bN_{n}$,};
\end{tikzpicture}\]
which is not a tree, but certainly (\ref{cond2}) has a solution (\ref{solution}) corresponding to $\psi = \mathrm{Id}_G = \pi$. Nevertheless, we do need $\Gamma_{\Psi,\Pi}$ to have no cycle.

\begin{prop}\label{no cycle prop}There is no cycle in $\Gamma_{\Psi,\Pi}$.
\end{prop}
\begin{proof}Suppose for contradiction that $\Gamma_{\Psi,\Pi}$ has a cycle $\mathfrak{c}$, whose length must be even, say $2r$ with $r\in\bN$, because $\Gamma_{\Psi,\Pi}$ is bipartite. Then $\mathfrak{c}$ has the shape
\[\begin{tikzpicture}
\node at (0,0) [name = 1] {$i_{1,\pi}$};
\node at (-3,-0.5) [name = 2] {$j_{1,\psi}$};
\node at (0,-1) [name = 3] {$i_{2,\pi}$};
\node at (-3,-1.5) [name = 4] {$j_{2,\psi}$};
\node at (0,-2.5) [name = 5] {$i_{r,\pi}$};
\node at (-3,-3) [name = 6] {$j_{r,\psi}$};
\node at (-1.5,-1.9) {$\vdots$};
\node at (-3,-2.15) {$\vdots$};
\node at (0,-1.65) {$\vdots$};
\draw[-] (1) -- (2) -- (3) -- (4) (5) -- (6) -- (1);
\draw[loosely dotted,-] (4) -- (5);
\node at (-9.5,-0.8) [align =center] {$\mathfrak{c}:\hspace{1mm}i_{1,\pi} -j_{1,\psi} - i_{2,\pi} - j_{2,\psi} - \cdots$\\\hspace{3.5cm}$\cdots- i_{r,\pi} - j_{r,\psi}-i_{1,\pi}$};
\end{tikzpicture}\]
where $i_1,\dots,i_r,j_1,\dots,j_r\in\bN_{n}$. For each $j\in \bN_n$, by Definition \ref{Gamma bipartite}, there are at most two edges in $\Gamma_{\Psi,\Pi}$ with $j_\psi$ as an endpoint, namely 
\[ \begin{cases}
\mathfrak{e}_{j,\mathbf{f}}:\hspace{1mm}j_\psi-(\theta_f\circ\theta_\psi)(j)_\pi & \mbox{defined when $(\theta_f\circ\theta_\psi)(j)\neq0$},\\
\mathfrak{e}_{j,\mathbf{g}}:\hspace{1mm} j_{\psi}-(\theta_g\circ\theta_\psi)(j)_\pi & \mbox{defined when $(\theta_g\circ\theta_\psi)(j)\neq0$}.\end{cases}\]
Putting $i_{r+1} = i_{1}$, we then deduce that 
\[ \{\theta_f(\theta_\psi(j_{s})), \theta_g(\theta_\psi(j_{s}))\} = \{i_{s},i_{s+1}\}\mbox{ for all }1\leq s\leq r.\]
But then by Definition \ref{graph fg def}, this implies that there is an edge joining $i_{s}$ and $i_{s+1}$ in $\Gamma_{\{f,g\}}$, so we have a cycle
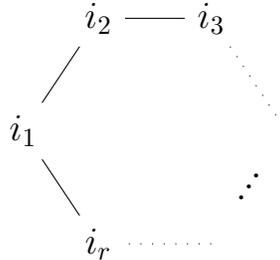
\[\begin{tikzpicture}
\node at (0,0) [name = 1] {$i_1$};
\node at (1,1.5) [name = 2] {$i_2$};
\node at (2.5,1.5) [name =3] {$i_3$};
\node at (1,-1.5) [name = 4] {$i_{r}$};
\draw[-] (1) -- (2) -- (3) (4) -- (1);
\draw[loosely dotted,-] (3) -- (3.5,0) (2.5,-1.5) -- (4);
\node at (3,-0.75) {$\cdot$};
\node at (3.1,-0.75+0.15) {$\cdot$};
\node at (2.9,-0.75-0.15) {$\cdot$};
\end{tikzpicture}\]
in $\Gamma_{\{f,g\}}$, which is impossible because $\Gamma_{\{f,g\}}$ is a tree. \end{proof}

For convenience, let us refine Definition \ref{Gamma bipartite} as follows.

\begin{definition}\label{def2}Define $\Gamma_{\Psi,\Pi,\leftrightarrow}$ to be the symmetric directed graph that is obtained from $\Gamma_{\Psi,\Pi}$ by replacing every edge by a pair of arrows going in the opposite directions. In other words, for each $i\in\bN_{n}$, we have the arrows
\[\begin{tikzpicture}
\node at (0,0) [name = 1] {$i_\psi$};
\node at (3,0) [name = 2] {\hspace{1cm}$(\theta_f\circ\theta_\psi)(i)_\pi$};
\path[->] (1) edge [bend right] node [below] {$\mathfrak{a}_{i,\mathbf{f}}$} (2);
\path[->] (2) edge [bend right] node [above] {$\mathfrak{a}_{i,\mathbf{f}}^{-1}$} (1);
\node at (6.5,0) {and};
\node at (8.75,0) [name = 1'] {$i_\psi$};
\node at (11.75,0) [name = 2'] {\hspace{1cm}$(\theta_g\circ\theta_\psi)(i)_\pi$};
\path[->] (1') edge [bend right] node [below] {$\mathfrak{a}_{i,\mathbf{g}}$} (2');
\path[->] (2') edge [bend right] node [above] {$\mathfrak{a}_{i,\mathbf{g}}^{-1}$} (1');
\end{tikzpicture}\]
in $\Gamma_{\Psi,\Pi,\leftrightarrow}$ when $i\not\in\bN_{n,\mathbf{f}}$ and $i\not\in\bN_{n,\mathbf{g}}$, respectively. 
\end{definition}

The next lemma is not necessary, but it makes the proof slightly cleaner.

\begin{lem}\label{singleton lem}There is no vertex in $\Gamma_{\Psi,\Pi}$ which has degree $0$.
\end{lem}
\begin{proof}For any $i\in \bN_n$, we cannot have $\theta_f(\theta_\psi(i))$ and $\theta_g(\theta_\psi(i))$ both equal to $0$ because $\Gamma_{\{f,g\}}$ has no loop, so plainly vertices $i_\psi$ in $\Psi$ have degree at least one. To show that vertices $j_\pi$ in $\Pi$ have degree at least one, it is enough to show that for any $j\in\bN_n$, one of the equations
\[ \theta_f(\theta_\psi(i)) = j\mbox{ and }\theta_g(\theta_\psi(i)) = j\]
has a solution $i\in\bN_n$. By (\ref{cond1}), they may be respectively rewritten as
\[ \theta_{f'}(i) = \theta_\pi(j) \mbox{ and }\theta_{g'}(i) = \theta_\pi(j).\]
Since $\Gamma_{\{f',g'\}}$ is connected, there must be an edge containing $\theta_\pi(j)$ in $\Gamma_{\{f',g'\}}$. This means that indeed one of these two equations is solvable in $i\in\bN_n$.
\end{proof}

We shall now complete the proof of the backward implication of Theorem \ref{key thm}. Recall that we have already picked bijections $\theta_\pi$ and $\theta_\psi$ for which (\ref{cond1}) is satisfied. It remains to exhibit a solution (\ref{solution}) to the equations (\ref{cond2f}) and (\ref{cond2g}). As illustrated in Example \ref{ex}, we shall do so by ``following the edges" in the graph $\Gamma_{\Psi,\Pi}$. More precisely, we shall carry out the following steps for each connected component $\Gamma$ of $\Gamma_{\Psi,\Pi}$. Notice that $\Gamma$ cannot be a singleton by Lemma \ref{singleton lem} and so it contains a vertex in $\Psi$.
\begin{enumerate}[1.]
\item Fix a vertex $i_{0,\psi}$ in $\Psi$ lying in $\Gamma$ and pick $\varphi_{\psi,i_0}\in\mathrm{Iso}(T^{(\theta_\psi(i_0))},T^{(i_0)})$ to be any isomorphism.
\item Let $v$ be any other vertex in $\Gamma$. Since $\Gamma$ is a tree by Proposition \ref{no cycle prop}, there exists a unique simple path in $\Gamma$ joining $i_{0,\psi}$ and $v$, which corresponds to a simple directed path $\mathfrak{p}$ in $\Gamma_{\Psi,\Pi,\leftrightarrow}$. Let us write $\mathfrak{p} = \mathfrak{a}\cdot \mathfrak{q}$ as a concatenation, where $\mathfrak{q}$ is a simple directed path starting at $i_{0,\psi}$, and $\mathfrak{a}$ is an arrow ending at $v$. Let $d\in\bN$ be the length of $\mathfrak{p}$, and $d-1$ would be the length of $\mathfrak{q}$.
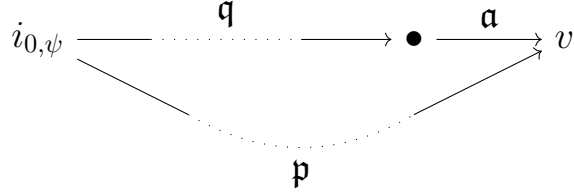
\[\begin{tikzpicture}
\node at (0,0) [name = 1] {$i_{0,\psi}$};
\node at (5,0) [name = 2] {$\bullet$};
\node at (7,0) [name = 3] {$v$};
\draw[-] (1) -- (1.5,0);
\path[-,loosely dotted] (1.5,0) edge node [above] {$\mathfrak{q}$} (3.5,0);
\draw[->] (3.5,0) -- (2);
\path[->] (2) edge node [above] {$\mathfrak{a}$} (3);
\draw[-] (1) -- (2,-1);
\draw[->] (5,-1) -- (3);
\path[-,loosely dotted] (2,-1) edge [bend right = 30] node [below] {$\mathfrak{p}$} (5,-1);
\end{tikzpicture}\]
There are two cases to consider.
\begin{enumerate}[(a)]
\item In the case that $d$ is odd, we have
\[v = j_{\pi} \mbox{ with } j = (\theta_h\circ\theta_\psi)(i)\mbox{ and }\mathfrak{a} = \mathfrak{a}_{i,\mathbf{h}}\]
for some $i,j\in\bN_n$ and $(h,\mathbf{h})\in\{(f,\mathbf{f}),(g,\mathbf{g})\}$, as illustrated below.
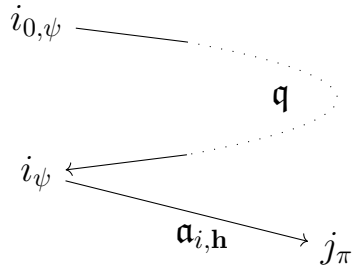
\[\begin{tikzpicture}
\node at (0,0) [name = 0] {$i_{0,\psi}$};
\node at (0,-2) [name = j] {$i_{\psi}$};
\node at (4,-3) [name = d] {$j_{\pi}$};
\draw[-] (0) -- (2,-0.25);
\draw[->] (2,-0.5-2.5/2) -- (j);
\node at (3.25,-1) {$\mathfrak{q}$};
\path[->] (j) edge node [below] {\hspace{4mm}$\mathfrak{a}_{i,\mathbf{h}}$} (d);
\draw[rotate=90,loosely dotted] (-0.25,-2) parabola bend (-1,-4) (-1.75,-2);
\end{tikzpicture}\]
Assuming that $\varphi_{\psi,i}\in\mathrm{Iso}(T^{(\theta_\psi(i))},T^{(i)})$ has been chosen, we define
\[ \varphi_{\pi,j} = \Phi_{i,\mathbf{h}}(\varphi_{\psi,i}),\mbox{ so then }\varphi_{h',i} = \varphi_{\psi,i}\circ\varphi_{h,\theta_\psi(i)}\circ\varphi_{\pi,(\theta_h\circ\theta_\psi)(i)}\]
is satisfied, as observed in (\ref{observation2}).
\item In the case that $d$ is even, we have
\[v = i_{\psi} \mbox{ with } j = (\theta_h\circ\theta_\psi)(i)\mbox{ and }\mathfrak{a} = \mathfrak{a}_{i,\mathbf{h}}^{-1}\]
for some $i,j\in\bN_n$ and $(h,\mathbf{h})\in\{(f,\mathbf{f}),(g,\mathbf{g})\}$, as illustrated below.
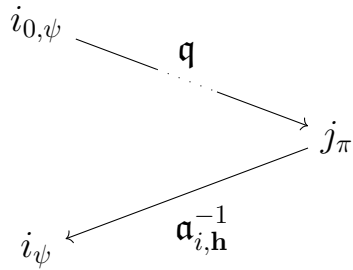
\[\begin{tikzpicture}
\node at (0,0) [name = 0] {$i_{0,\psi}$};
\node at (4,-1.5) [name = j] {$j_{\pi}$};
\node at (0,-3) [name = d] {$i_{\psi}$};
\node at (2,-0.4) {$\mathfrak{q}$};
\draw[loosely dotted,-] (8/5,-0.6) -- (12/5,-0.9);
\draw[-] (0) -- (8/5,-0.6);
\draw[->] (12/5,-0.9) -- (j);
\path[->] (j) edge node [below] {\hspace{4mm}$\mathfrak{a}_{i,\mathbf{h}}^{-1}$} (d);
\end{tikzpicture}\]
Assuming that $\varphi_{\pi,j}\in\mathrm{Iso}(T^{(\theta_\pi(j))},T^{(j)})$ has been chosen, we define
\[ \varphi_{\psi,i} = \Phi^{-1}_{i,\mathbf{h}}(\varphi_{\pi,j}),\mbox{ so then }\varphi_{h',i} = \varphi_{\psi,i}\circ\varphi_{h,\theta_\psi(i)}\circ\varphi_{\pi,(\theta_h\circ\theta_\psi)(i)}\]
is satisfied, as observed in (\ref{observation2}).
\end{enumerate}
In either case, the equation (\ref{cond2f}) or (\ref{cond2g}) represented by $\mathfrak{a}$ is satisfied.
\end{enumerate}
Therefore, by induction on $d$, the above allows us to define $\varphi_{\psi,i}$ and $\varphi_{\pi,j}$ for all vertices $i_\psi\in\Psi\cap\Gamma$ and $j_\pi\in\Pi\cap\Gamma$, such that the equations  (\ref{cond2f}) or (\ref{cond2g}) represented by the edges in $\Gamma$ are all satisfied. Applying this to all connected components of $\Gamma_{\Psi,\Pi}$ then gives us  a solution (\ref{solution}) to the equations (\ref{cond2f}) and (\ref{cond2g}). This completes the proof.
 
\section{Counting unlabeled directed trees with a root} \label{last sec} 
 
Let $\mathcal{T}^\star(n)$ be the set defined as in Remark \ref{rem} and write $t_n^\star$ for its size. We shall end this paper by briefly discussing the problem of computing $t_n^\star$. 

\vspace{1mm}

Clearly $t_0^\star = 1$ and $t_1^\star=2$. It is also easy to see that $t_2^\star = 7$, with
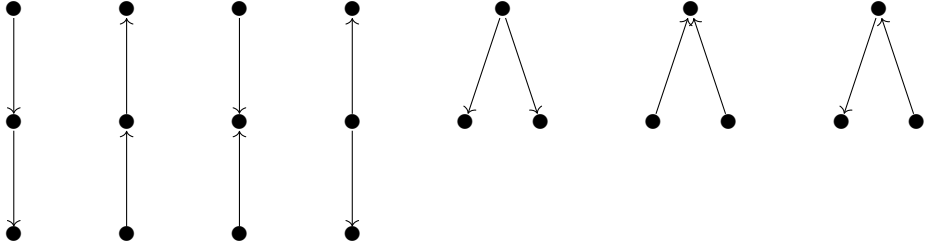
\[\begin{tikzpicture}[inner sep=0pt, outer sep=0pt]
\node at (0,0) [name=r1] {$\bullet$};
\node at (0,-1.5) [name=v11] {$\bullet$};
\node at (0,-3) [name=v21] {$\bullet$};
\node at (1.5,0) [name=r2] {$\bullet$};
\node at (1.5,-1.5) [name=v12] {$\bullet$};
\node at (1.5,-3) [name=v22] {$\bullet$};
\node at (3,0) [name=r3] {$\bullet$};
\node at (3,-1.5) [name=v13] {$\bullet$};
\node at (3,-3) [name=v23] {$\bullet$};
\node at (4.5,0) [name=r4] {$\bullet$};
\node at (4.5,-1.5) [name=v14] {$\bullet$};
\node at (4.5,-3) [name=v24] {$\bullet$};
\node at (6.5,0) [name=r5] {$\bullet$};
\node at (6,-1.5) [name=v15] {$\bullet$};
\node at (7,-1.5) [name=v25] {$\bullet$};
\node at (9,0) [name=r6] {$\bullet$};
\node at (8.5,-1.5) [name=v16] {$\bullet$};
\node at (9.5,-1.5) [name=v26] {$\bullet$};
\node at (11.5,0) [name=r7] {$\bullet$};
\node at (11,-1.5) [name=v17] {$\bullet$};
\node at (12,-1.5) [name=v27] {$\bullet$};
\draw[->] (r1) -- (v11);
\draw[->] (v11) -- (v21);
\draw[->] (v22) -- (v12);
\draw[->] (v12) -- (r2);
\draw[->] (r3) -- (v13);
\draw[->] (v23) -- (v13);
\draw[->] (v14) -- (r4);
\draw[->] (v14) -- (v24);
\draw[->] (r5) -- (v15);
\draw[->] (r5) -- (v25);
\draw[->] (v16) -- (r6);
\draw[->] (v26) -- (r6);
\draw[->] (r7) -- (v17);
\draw[->] (v27) -- (r7);
\end{tikzpicture}\]
as the elements of $\mathcal{T}^\star(2)$, where the root is being placed at the top of graph. For larger values of $n$, we can somewhat use induction to compute $t_n^\star$. 

\begin{definition}Let $\Gamma$ be any unlabeled directed tree with a root and write $\Gamma_{\mathrm{un}}$ for its underlying undirected graph. Let $0$ denote the root, and let $v$ be any other vertex of $\Gamma$.
\begin{enumerate}[$(1)$]
\item Let $\Gamma_v$ be the unlabeled directed tree with root $v$ which is obtained from $\Gamma$ by removing all the vertices $v'$ such that the unique simple path in $\Gamma_{\mathrm{un}}$ joining $0$ and $v'$ does not contain $v$. We call $\Gamma_v$ the \emph{subtree} of $\Gamma$ at $v$.
\item The vertex $v$ is a \emph{child} of $\Gamma$ if the unique simple path in $\Gamma_{\mathrm{un}}$ joining $0$ and $v$ has length one. In this case, this path corresponds to an arrow $\mathfrak{a}$ in $\Gamma$. We call $v$ an \emph{in-child} if $\mathfrak{a}$ ends at $0$, and an \emph{out-child} if $\mathfrak{a}$ starts at $0$.
\end{enumerate}
\end{definition}

\begin{prop}\label{count prop}Suppose that we are given tuples
\[(n_1,n_2),\,\ (k_{1,1}\dots,k_{1,n_1}),\,\ (k_{2,1},\dots,k_{2,n_2})\]
of non-negative integers such that
\[ n = n_1 + n_2,\,\  n_1 = \sum_{s=1}^{n_1} sk_{1,s},\,\ n_2 = \sum_{s=1}^{n_2}sk_{2,s},\]
where an empty sum represents $0$. Then, the number of $\Gamma\in\mathcal{T}^\star(n)$ such that
\begin{enumerate}[$(a)$]
\item For each $1\leq s\leq n_1$, the number of in-children $v$ whose subtrees $\Gamma_v$ have \par\noindent exactly $s$ vertices is equal to $k_{1,s}$;
\item For each $1\leq s\leq n_2$, the number of out-children $v$ whose subtrees $\Gamma_v$ have exactly $s$ vertices is equal to $k_{2,s}$;\end{enumerate}
are both satisfied, is given by
\[ \prod_{s=1}^{n_1}{t_{s-1}^\star + k_{1,s} - 1\choose k_{1,s}} \cdot \prod_{s=1}^{n_2}{t_{s-1}^\star + k_{2,s} - 1\choose k_{2,s}},\]
where an empty product represents $1$.
\end{prop}
\begin{proof}A similar statement is true for unlabeled rooted trees; see \cite[p. 386]{graph} for example. The only difference is that our trees are directed, which means that the in-children and out-children have to be treated differently.

\vspace{1mm}

To count the number of $\Gamma\in\mathcal{T}^\star(n)$ satisfying (a) and (b), note that $\Gamma$ may be viewed as a collection of subtrees $\Gamma_v$ at the children $v$. Also, the subtrees $\Gamma_v$ at the in-children $v$ are independent from those at the out-children $v$. For each $\ell=1,2$ and $1\leq s\leq n_\ell$, at the $k_{\ell,s}$ vertices $v$ which are 
\[\begin{cases}
\mbox{in-children}&\mbox{if }\ell=1\\
\mbox{out-children}&\mbox{if }\ell=2
\end{cases}\]
and whose subtrees $\Gamma_v$ are to have exactly $s$ vertices, we can choose any $k_{\ell,s}$ elements of $T^\star(s-1)$ to be these $\Gamma_v$ with repetition allowed. There are
\[ \multiset{t_{s-1}^\star}{k_{\ell,s}}  = {t_{s-1}^\star+ k_{\ell,s}-1\choose k_{\ell,s}}\]
such choices. Multiplying over $s$ and $\ell$ then yields the claim.
\end{proof}

Given any $m\in\bN$, let $\mathcal{A}(m)$ be the set of all $m$-tuples $\mathbf{k} = (k_1,\dots,k_m)$ of non-negative integers such that $m = k_1 + 2k_2 + \cdots + mk_m$, and put
\[ P(m;\mathbf{k}) = \prod_{s=1}^{m}{t_{s-1}^\star+k_s-1\choose k_s}\]
in this case. From Proposition \ref{count prop}, we then deduce that
\begin{equation}\label{tn formula}
  t_n^\star = \sum_{m\in\bN_{0,n}}\sum_{\mathbf{k}_1\in \mathcal{A}(m)}\sum_{\mathbf{k}_2\in \mathcal{A}(n-m)}P(m;\mathbf{k}_1)P(n-m;\mathbf{k}_2).\end{equation}
This gives us a way to compute $t_n^\star$ inductively, but the calculation gets very complicated. The better approach is perhaps to consider the associated generating function given by
\[ z\mapsto \sum_{n=0}^{\infty} t_n^\star z^{n},\]
and then use (\ref{tn formula}) to study its properties. But this is a completely different problem. We shall therefore content ourselves with computing $t_n^\star$ for $n\leq 10$ by applying the formula (\ref{tn formula}) directly. 

\vspace{2mm}

\begin{center}\begin{tabular} {|c|c|}
\hline
$n$ & $t_n^\star$\\\hline\hline
$1$ & $1$\\\hline
$2$ & $7$\\\hline
$3$ & $26$\\\hline
$4$ & $107$\\\hline
$5$ & $458$\\\hline
\end{tabular}\quad\quad\quad\quad
\begin{tabular} {|c|c|}
\hline
$n$ & $t_n^\star$\\\hline\hline
$6$ & $2058$\\\hline
$7$ & $9498$\\\hline
$8$ & $44947$\\\hline
$9$ & $216598$\\\hline
$10$ & $1059952$\\\hline
\end{tabular}
\end{center}

\vspace{2mm}

The calculation was done in \textsc{Magma} \cite{magma}. For $1\leq n\leq 7$, the computation only took a few seconds in total. However, the set $\mathcal{A}(n)$ gets very difficult to compute very quickly. For $n\geq 8$, the code gets very slow, and we were only able to run it up to $n=10$. We note that the skew braces (for which the additive and circle groups are isomorphic to the finite same characteristically simple group) enumerated by this calculation are not covered by \cite{Skew braces} and \cite{Skew braces 2} mentioned in the introduction.

\vspace{1mm}

Finally, the code used is a straightforward implementation of (\ref{tn formula}), but we include it below for the interested reader.

{\small \begin{lstlisting}
t:=[1]; // The $(n+1)$st term in t is $t_n^\star$.
A:=[[]]; // The $(n+1)$st term in A is the set $\mathcal{A}(n)$ (except when $n=0$).
P:=[[1]]; // The $(n+1)$st term in P is the set $\mathcal{P}(n)$.
// Suppose we wish to compute $t_1^\star,\dots,t_\ell^\star$.
for n in [1..$\ell$] do 
// Compute $\mathcal{A}(n)$ and add it to the list A.
  Cn:=CartesianPower([0..n],n);
  An:=[Tuplist(k):k in Cn|n eq &+[s*k[s]:s in [1..n]]];
  Append(~A,An);
// Compute $\mathcal{P}(n)$ and add it to the list P.
  Pn:=[&*[Binomial(t[s]+k[s]-1,k[s]):s in [1..n]]:k in An];
  Append(~P,Pn);
// Compute $t_n^\star$ using formula (5.1) and add it to the list t.
  tn:=&+[&+[P1*P2:P1 in P[m+1],P2 in P[n-m+1]]:m in [0..n]];
  Append(~t,tn);
// Output $n$ and $t_n^\star$.
<n,tn>;
end for;
// t; // The new list t now consists of $t_0^\star,\dots,t_\ell^{\star}$.
// A; // The new list A now consists of [],$\mathcal{A}(1),\dots,\mathcal{A}(\ell)$.
// P; // The new list P now consists of $\mathcal{P}(0),\dots,\mathcal{P}(\ell)$.
// Replacing t, A, and P with these new lists, we may run the same
code for n in [$\ell+1$..$\ell'$] to compute $t_{\ell+1}^\star,\dots,t_{\ell'}^\star$ without having to recompute 
the data for $n$ up to $\ell$.
\end{lstlisting}}

\section*{Acknowledgments} 



The author thanks Dr. Chao Qin (Chris King) for helping her calculate $t_n^\star$ for $n\leq 10$  by running the \textsc{Magma} code on his computer. She would also like to thank the referee for helpful comments.

\end{document}